\documentclass{article}
\usepackage{a4, amsthm, amsmath, amssymb,amsbsy, epic,graphicx,mathrsfs,enumerate}
\usepackage[all]{xy}
\usepackage{multicol,longtable}

\newtheorem{thm}{Theorem}[section]
\newtheorem{lem}[thm]{Lemma}
\newtheorem{cor}[thm]{Corollary}
\newtheorem{prop}[thm]{Proposition}
\newtheorem{defi}[thm]{Definition}

\theoremstyle{definition}
\newtheorem*{remark}{Remark}
\newtheorem{exam}{Example}[section]

\newcommand{\coker}{\operatorname{coker}}
\newcommand{\im}{\operatorname{im}}

\newcommand{\Aut}{\operatorname{Aut}}

\newcommand{\rat}{\mathbb{Q}}                          
\newcommand{\ra}{\rightarrow}
\newcommand{\mZ}{\mathbb{Z}}
\newcommand{\mF}{\mathbb{F}}
\newcommand{\mQ}{\mathbb{Q}}

\newcommand{\mN}{\mathbb{N}}
\newcommand{\Ext}{\operatorname{Ext}}
\newcommand{\Hom}{\operatorname{Hom}}
\newcommand{\Ann}{\operatorname{Ann}}

\newcommand{\GL}{\operatorname{GL}}

\renewcommand{\rm}{\mathrm}

\newcommand{\mbf}{\mathbf}
\newcommand{\Cal}{\mathcal}

\newcommand{\hra}{\hookrightarrow}

\newcommand{\Fix}{\operatorname{Fix}}

\newcommand{\lda}{\lambda}
\renewcommand{\gcd}{\operatorname{gcd}}
\newcommand{\fr}{\mathfrak}
\newcommand{\bs}{\boldsymbol}
\newcommand{\scr}{\mathscr}
\newcommand{\lra}{\longrightarrow}
\newcommand{\id}{\operatorname{id}}
\newcommand{\ol}{\overline}

\title{Higman's PORC conjecture for a family of groups}
\author{Anton Evseev}
\date{}

\begin{document}

\maketitle

\abstract{We prove that the number of groups of order $p^n$ whose Frattini subgroup is central is for fixed $n$ a PORC (`polynomial on residue classes') function of $p$. This extends a result of G. Higman.}

\section{Introduction}

A rational-valued function $f(x)$ defined on a set $D$ of integers 
is said to be~\emph{polynomial} on $D$ if there exists a polynomial $P$ with
rational coefficients such that $f(x)=P(x)$ for all $x\in D$. 
A function $f(x)$ defined on a set $E$ of integers is said to be~\emph{PORC} on $E$ (`polynomial on residue classes') if there exists a positive integer $N$ such that $f(x)$ is polynomial on the intersection of $E$ with each residue class mod $N$; that is, $f(x)$ is polynomial on the sets 
$\{j+lN: l\in \mZ\}\cap E$, $j=0,1,\ldots,N-1$. 

If $p$ is a prime and $n\in \mN$, denote 
 the number of isomorphism classes of groups of order $p^n$ by $G_n (p)$. 
Table~\ref{table1} summarises most of the current knowledge of the values of $G_n (p)$ (see \cite{p6,p7} and references therein for details). The most recent 
result in this direction is an enumeration of all groups of order $p^7$ by E.R.~O'Brien and M.R.~Vaughan-Lee \cite{p7}.
\begin{table}
\caption{The values of $G_n (p)$ for $n\le 7$ 
and large enough $p$ (as specified)}\label{table1}
$$
\begin{array}{|c|l|}
\hline
n & \multicolumn{1}{c|}{G_n (p)} \\
\hline
1 & 1 \\
\hline
2 & 2 \\
\hline
3 & 5 \\
\hline
4 & 15 \text{ if } p\ge 3 \\
\hline
5 & 2p+61+2\gcd(p-1,3)+\gcd(p-1,4) \text{ if } p\ge 5 \\
\hline
6 & 3p^2+39p+344+24\gcd(p-1,3)+11\gcd(p-1,4)+2\gcd(p-1,5) \text{ if } p\ge 5 \\
\hline
7 &
\!\!\! \begin{array}{l}3p^5 + 12p^4 + 44p^3 + 170p^2 + 707p + 2455 + 
(4p^2+44p+291)\gcd(p-1,3) \\
+ (p^2+19p+135)\gcd(p-1,4)+(3p+31)\gcd(p-1,5)+4\gcd(p-1,7) \\ 
+5\gcd(p-1,8)+\gcd(p-1,9) \text{ if } p>5
\end{array} \\
\hline
\end{array}
$$
\end{table}
Note that $G_n (p)$ is a PORC function of $p$ for $n\le 7$.
G.~Higman conjectured in 1960 that, for a fixed $n$, $G_n (p)$ is PORC 
(see~\cite{Higman1, Higman}).

We shall concentrate on groups of order $p^n$ of nilpotency class 2 for an arbitrary $n$.  G.~Higman~\cite{Higman} proved that the number of groups of order $p^n$ whose Frattini subgroup is elementary abelian and central is a PORC function of $p$. In this paper we extend Higman's theorem by proving the following 
result.

\begin{thm} \label{PORCthm} For a fixed natural number $n$, the number of isomorphism classes of groups of order $p^n$ whose Frattini subgroup is central, considered as a function of the prime $p$, is PORC.\end{thm}

\begin{remark} The groups counted in Theorem~\ref{PORCthm} are precisely the groups $G$ of order $p^n$ such that $[x,y^p]=1$ and $[[x,y],z]=1$ for all 
$x,y,z\in G$.  
\end{remark}

This paper is organised as follows. 
In Section~\ref{set-up} we use Lazard correspondence to restate 
Theorem~\ref{PORCthm} in terms of Lie algebras. 
We also define the concepts used in the proof. 
In Section~\ref{types} we prove a number of basic properties of those concepts.  Section~\ref{higthm} contains a generalisation of \cite[Theorem 1.2.1]{Higman}, which is key to the proof of Theorem~\ref{PORCthm}. Section \ref{Hallpoly} 
develops corollaries from the existence of Hall polynomials 
and from similar results. 
In Section~\ref{autsection} we investigate automorphisms of finite modules over discrete valuation rings. In Section~\ref{Extsection} we determine the functor
$\Ext^1 (-,-)$ for finitely generated modules over principal ideal domains up to natural equivalence. 
Finally, in Section~\ref{finproof}, we put all these 
results together to prove Theorem~\ref{PORCthm}. 

\vspace{0.8cm}

\textit{Notation and definitions}

\begin{itemize}
\item $\gamma(G,X)$ denotes the number of orbits of an action of a group $G$ on a finite set $X$, where the action is understood;
\item $[k,n]=\{k,k+1,k+2,\ldots,n\}$ where $k \le n$ are integers;
\item $|X|$ denotes the cardinality of a set $X$;
\item $\id_X$ is the identity map $X\ra X$;
\item $\delta_{ij}=0$ if $i\ne j$, and $\delta_{ii}=1$;
\item $\mF_q$ is the field with $q$ elements (if $q$ is a prime power);
\item $\mZ_p$ is the ring of $p$-adic integers (if $p$ is a prime);
\item If $x,y$ are elements of a group, then $[x,y]=x^{-1}y^{-1}xy$;  
\item $Z(G)$ is the centre of a group $G$;
\item $C_G (g)=\{h\in G: gh=hg \}$ is the centraliser of an element $g$ of a group $G$;
\item A $G$-\emph{set} is a set $X$ together with an action of $G$ on $X$;
\item $\phi:X \ra Y$ is an \emph{isomorphism of $G$-sets} $X$ and $Y$ if $\phi$
is a bijection and $\phi(gx)=g(\phi(x))$ for all $x\in X$, $g\in G$;
\item $Z(L)$ denotes the centre of a Lie algebra $L$;
\item If $A$ is a module over a commutative ring containing $t$, then $A[t]=\{a\in A:\; ta=0\}$;
\item $\GL(V)$ is the group of invertible linear transformations of a vector space $V$;
\item $\scr P(V;U_1,\ldots,U_k):=\{ f\in \GL(V): f(U_i)\subseteq U_i \; \forall i \}$ where $U_i$ are subspaces of $V$;
\item $\GL_n (K)=\GL(K^n)$ if $K$ is a field;
\item $\GL_n (q)=\GL_n (\mF_q)$ where $q$ is a prime power;
\item $U\le V$: $U$ is a subspace of a vector space $V$;  
\item $U<V$: $U$ is a subspace of $V$ and $U\ne V$;
\item $V^*$ is the dual space of a vector space $V$;
\item If $R$ is a commutative ring and $A$ is an $R$-module, then 
$\Aut(A)$ is the group of automorphisms $A$;
\item $(t)=tR$ is the principal ideal generated by an element 
$t$ of a commutative ring $R$; 
\item $R[X]$ is the ring of polynomials in one variable over a commutative ring $R$;
\item $R[[X]]$ is the ring of formal power series in one variable in $R$;
\item $\Hom_R (A,B)$ is the $R$-module of homomorphisms between
$R$-modules $A$ to $B$;
\item $\Ann(a):=\{r\in R: ra=0 \}$ is the annihilator of an element $a$ of 
an $R$-module $A$;
\item  If $\tau:A \ra B$ is an $R$-module homomorphism, then 
$\coker\tau=B/\im\tau$.
\item $\bigwedge\nolimits^2 A$ is the exterior square of a module $A$;
\item A \emph{partition} is a finite (possibly, empty) 
sequence $\lda=(\lda_1,\ldots,\lda_s)$ of
positive integers such that $\lda_1\ge \cdots \ge \lda_s$;
\item $|\lda|:=\lda_1+\lda_2+\cdots+\lda_s$;
\item $()$ is the partition with no parts.
\end{itemize}

\textbf{Acknowledgements.} This work is a part of my D.Phil. thesis. 
I am grateful to my supervisor, Marcus du Sautoy, for suggesting the problem and 
 his continuous support and to my thesis examiners, Dan Segal and Gerhard R\"ohrle, 
for spotting numerous mistakes and misprints. The remaining errors are, of course, my responsibility.

\section{The Lazard correspondence and types} 
\label{set-up}  

Elementary properties of PORC functions are summarised in the following lemma.

\begin{lem}\label{PORClem} Let $f(x)$, $g(x)$ and $h(x)$ be functions defined on a subset $D$ of integers. Then
\begin{enumerate}[(i)]
\item\label{sumporc} 
if $f$ and $g$ are PORC on $D$, then so are $f+g$ and $fg$;
\item \label{quot} if $f(x)$ and $g(x)$ are polynomial (PORC) on $D$, $h(x)=f(x)/g(x)$ for all $x\in D$ and $h(x)$ takes only integral values, then $h(x)$ is polynomial (respectively, PORC) on $D$;
\item\label{finporc}  if $D$ is the set of all prime numbers, $f$ is PORC on $D$ and $g(p)$ differs from $f(p)$ only for finitely many primes $p$, then $g$ is PORC on $D$.
\end{enumerate}
\end{lem}
\begin{proof} \eqref{sumporc} is straightforward. The polynomial case of \eqref{quot}
is well known; the `PORC' case is~\cite[Lemma 1.1.4]{Higman}. For 
\eqref{finporc}, let $E$ be a finite subset of $D$ such that $f$ and $g$ coincide on $D\setminus E$. Since $f$ is PORC on $D$, there exists $N\in \mN$ such 
that $f$ is polynomial on the intersection of $D$ with each residue class 
modulo $N$. Let $N'=N\prod_{p\in E} p$. Then $g$ is polynomial on the intersection of $D$ with each residue class modulo $N'$.  
\end{proof} 
  
For $p\ne 2$, there is a well-known one-to-one (Lazard) correspondence between finite $p$-groups of nilpotency class at most $2$ and finite Lie algebras over $\mZ_{p}$ of the same class. Namely, if $G$ is such a group, the corresponding Lie algebra $L$ has the same elements of $G$ as a set; addition in $L$ is defined as 
$$ 
x+y=xy[y,x]^{1/2}
$$
(if $z\in G$, $z^{1/2}:=z^{(p^k+1)/2}$ where $p^k$ is the order of $z$),
and the Lie bracket of $L$ is simply the commutator operation on $G$. (A similar correspondence may also be defined for groups and Lie rings of higher classes via the famous Campbell-Baker-Hausdorff formula.) The set $Z(G)$ is equal to the set $Z(L)$, and the Frattini subgroup $[G,G]G^p$ of $G$ coincides, as a set, with the ideal $pL+[L,L]$ of $L$.  
Therefore, Theorem \ref{PORCthm} is implied by the following result. 

\begin{thm} \label{Lie} For a fixed integer $n>0$, the number of isomorphism classes of Lie algebras $L$ of order $p^n$ satisfying $pL+[L,L]\subseteq Z(L)$ is a PORC function of the prime $p$.
\end{thm}

The remainder of this paper is devoted to a proof of Theorem \ref{Lie}. We begin with the necessary definitions, mostly following those of \cite{Higman}.

Let $\mbf n=(n_1,\ldots,n_r)$ be a tuple of nonnegative integers. Let $\GL_{\mbf n}(K)$ be the direct product of the groups $\GL_{n_i} (K)$, $i=1,\ldots, r$. Let $\Gamma$ be an algebraic morphism from $\GL_{\mbf n}(\mQ)$ 
into $\GL_m(\mQ)$. That is, if $\mathbf{g}=(g_1,\ldots,g_r)$ and $\mathbf{h}=(h_1,\ldots,h_r)$ are elements of $\GL_{\mbf n}(\rat)$, then
\begin{enumerate}[(i)]
\item $\Gamma(\mathbf{g})\Gamma(\mathbf{h})=\Gamma(\mathbf{gh})$, and 
\item  the entries of $\Gamma(\mathbf{g})$ are (fixed) rational functions of the entries of $g_1,\ldots,g_r$.
\end{enumerate}
It follows that  
$$\Gamma(\mathbf{g})=\frac{\Gamma_0(\mathbf{g})}{(\det g_1)^{s_1}\dots (\det g_r)^{s_r}},$$ 
where $\Gamma_0$ is a morphism whose elements are polynomials of those of $\mathbf{g}$, and $s_1,\dots,s_r$ are suitable nonnegative integers 
(see \cite[Section 1.2]{Higman}). 
The polynomials defining the elements of $\Gamma_0(\mathbf{g})$ have rational coefficients which can be brought over a common denominator $d$, say. If $K$ is a field whose characteristic does not divide $d$, these coefficients can be interpreted as elements of $K$, and hence $\Gamma$ as a homomorphism of $\GL_{n_1,\dots,n_r}(K)$ into $\GL_m (K)$. We denote the image of this homomorphism by $\Gamma(K)$. We shall also use the symbol $\Gamma(K)$ to denote the corresponding group of linear transformations of a vector space of dimension $m$ over $K$. The collection of linear groups $\Gamma(K)$ obtained in this way is called an \emph{algebraic family of groups}.

Let $k\ge 0$ be an integer, and let $K$ be a field. Let $V$ be a $k$-dimensional vector space over $K$. If $\lda=(\lda_1,\ldots,\lda_s)$ is a partition and $f\in K[X]$, define the $K[X]$-module $M_{\lda}(f)$ by 
$$M_{\lda}(f) = \frac{K[X]}{(f^{\lda_1})} \oplus 
\frac{K[X]}{(f^{\lda_2})}\oplus \cdots
\oplus \frac{K[X]}{(f^{\lda_s})}.
$$
An element $g$ of $\GL(V)$ defines a $K[X]$-module structure on $V$ in the usual way:
$$X\cdot v = gv \quad \forall v\in V.$$
 Since $K[X]$ is a principal ideal domain, this module is isomorphic to a module of the form
$$
M_{\lambda^1}(f_1)\oplus M_{\lambda^2}(f_2)\oplus \dots \oplus M_{\lambda^s}(f_s)
$$
where $\lambda^i$ are non-empty partitions and $f_i\in K[X]$ are distinct irreducible polynomials. We shall say that the polynomials $f_1,\ldots,f_s$ \emph{appear} 
in $g$. We now define the notion of a type, which encodes the combinatorial data of this decomposition that 
do not depend on the field $K$, namely the partitions $\lda^1,\ldots,\lda^s$ and the degrees of the polynomials $f_i$. 

Let $\mbf n=(n_1,\ldots, n_r)$ be a tuple of positive integers. 
An \emph{$\mbf n$-pretype} is a triple $\Cal P=(Q, \mbf d, \bs\lda)$ such that
\begin{enumerate}[(i)]
\item $Q$ is a finite (indexing) set;
\item $\mbf d$ is a tuple $(d_j)_{j\in Q}$ of positive integers;
\item 
$\bs\lda=(\lda^{ij})_{i\in [1,r],j\in Q}$, where each $\lda^{ij}$ is a partition; 
\item for all $i\in [1,r]$, $\sum_{j\in Q} d_j |\lda^{ij}|=n_i$.
\item for each $j\in Q$, there exists $i\in [1,r]$ such that $\lda^{ij}\ne ()$.
\end{enumerate}

An \emph{isomorphism} between two $\mbf n$-pretypes $(Q,\mbf d, \bs\lda)$ and $(Q',\mbf d',\bs\mu)$ is a bijection $f:Q\ra Q'$ such that
$d'_{f(j)}=d_j$ and $\mu^{i,f(j)}=\lda^{ij}$ for all $j\in Q$, $i\in [1,r]$. 
The \emph{automorphism group} $\Aut(\Cal P)$ of a pretype $\Cal P$ is the group of all isomorphisms from $\Cal P$ onto itself. 

An \emph{$\mbf n$-type} is an isomorphism class of $\mbf n$-pretypes. If 
$(Q,\mbf d,\bs\lda)$ is a pretype, we shall denote the corresponding type by
$[Q,\mbf d,\bs\lda]$.  

Now let $\mbf g=(g_1,\ldots,g_r)\in \GL_{\mbf n}(K)$. Then there 
exist polynomials $f_1,\ldots,f_s \in K[X]$ and a tuple of partitions 
$\bs\lda=(\lda^{ij})_{i\in [1,r], j\in [1,s]}$ such that
\begin{enumerate}[(i)]
\item for every $i\in [1,r]$, the $K[X]$-module given by $g_i$ is isomorphic to
$$M_{\lda^{i1}}(f_1)\oplus M_{\lda^{i2}}(f_2) \oplus \cdots 
\oplus M_{\lda^{is}}(f_s);
$$ 
\item\label{ne} for every $j\in [1,s]$, there exists $i\in [1,r]$ such that 
$\lda^{ij}\ne ()$.    
\end{enumerate}
(Note that condition (\ref{ne}) ensures that each polynomial $f_1,\ldots,f_s$ appears in at least one of $g_1,\ldots,g_r$.) 
Let $\mbf d = (d_1,\ldots,d_s)=(\deg f_1,\ldots,\deg f_s)$. Then 
$\Cal T:=[[1,s],\mbf d,\bs\lda]$ is an $\mbf n$-type, uniquely determined by 
$\mbf g$. We shall refer to it as the \emph{type} of $\mbf g$.

Any two conjugate elements of $\GL_{\mbf n}(K)$ are of the same type. If $C$ is a conjugacy class of the group $\GL_{\mbf n}(K)$, let the \emph{type} of $C$ be the type $\Cal T$ of an arbitrary element of $C$. We shall say that $C$ is a 
\emph{class of type} $\Cal T$.

\begin{remark} The above definition of a type is slightly different from that given by G.~Higman~\cite{Higman} (we shall refer to a type
 as defined in~\cite{Higman} as an \emph{H-type}). The two notions are 
equivalent when $\mbf n=(n)$ for some $n\in \mN$. However, the H-type of an element $\mbf g=(g_1,\ldots,g_r)\in\GL_{\mbf n}(q)$ is determined by the types of 
$g_1,\ldots,g_r$. This is not generally the case for the type of $\mbf g$.  
\end{remark}

Suppose that for each prime (or for each prime power) $q$ a subgroup 
$H=H(q)$ of $\GL_{\mbf n}(q)$ is specified. 
We shall call $H$ a \emph{uniform} (respectively, \emph{PORC}) 
\emph{family of subgroups} if for any
 $\mbf n$-type $\mathcal{T}$, for every class $C$ of this type, 
$|C\cap H(q)|$ depends only on $\mathcal{T}$ and $q$ (but not on $C$) and, for a fixed $\mathcal{T}$, is a polynomial (respectively, PORC) function of $q$.

\section{Properties of types}\label{types}

Let $\Cal P=(Q,\mbf d,\bs \lda)$ be an $\mbf n$-pretype, where 
$\mbf n=(n_1,\ldots,n_r)$ as before. 
Let $\Cal T$ be the corresponding type.
Say that a map $F:Q\ra \mF_q[X]$ is a \emph{realisation} of $\Cal P$ over $q$
 if $F$ is injective, $F(j)$ is irreducible and $\deg F(j)=d_j$ 
for all $j\in Q$. 
Let 
$\Cal R_{\Cal P} (q)$ be the set of all realisations of $\Cal P$ over $q$. Every such realisation $F$ gives rise to a conjugacy class $C(F)=C_{\Cal P}(F)$ 
of type $\Cal T$ 
as follows: $C(F)$ consists of all the tuples 
$\mbf g\in \GL_{\mbf n}(q)$ such that, for each 
$j\in [1,r]$, $g_i$ gives rise to a module isomorphic to 
$$
\bigoplus_{j\in Q} M_{\lda^{ij}}(F(j)).
$$   
Clearly, each conjugacy class of type $\Cal T:=[Q,\mbf d,\bs\lda]$ is of the form $C(F)$ for some 
$F\in \Cal R_{\Cal P}(q)$. 
\begin{lem}\label{real} Let $\Cal P=(Q,\mbf d,\bs\lda)$ be an $\mbf n$-pretype.
Let $C$ be a conjugacy class in $\GL_{\mbf n}(q)$ of type $[Q,\mbf d,\bs\lda]$. Then 
$$
|\{F\in \Cal R_{\Cal P}(q) : C(F)=C \} | = |\Aut(\Cal P)|.
$$
\end{lem}
\begin{proof} Let $F$ and $F'$ be realisations of $\Cal P$ over $\mF_q$. Then 
$C(F)=C(F')$ if and only if, for each $j\in Q$, there exists $k\in Q$ such that
$F(j)=F'(k)$ (so $d_j=d_k$) and, for all $i\in [1,r]$, 
$\lda^{ij}=\lda^{ik}$. This occurs if and only if $\im F'=\im F$ and 
$(F')^{-1}\circ F$ is an automorphism of $\Cal P$. The result follows. \end{proof} 

\begin{lem}\label{numcl} Fix an $\mbf n$-type $\Cal T$. The number of conjugacy classes 
$C$ in $\GL_{\mbf n}(q)$ of type $\Cal T$ is polynomial in $q$. 
\end{lem}
\begin{proof} By Lemma~\ref{real}, it is enough to show that the number of realisations of a pretype corresponding to $\Cal T$ is polynomial. However, this follows from the well-known fact that the number of irreducible polynomials of a fixed degree over $\mF_q$ is polynomial in $q$.
\end{proof}

Let $\Cal T=[Q,\mbf d,\bs\lda]$ be an $\mbf n$-type. 
Let $l\in [1,r]$, and let 
$\mbf m=(n_1,\ldots,n_l)$. The \emph{projection} of 
$\Cal T$ on the first $l$ components is the $\mbf m$-type $[R,\mbf e,\bs\mu]$ with
\begin{enumerate}[(i)]
\item $R=\{ j\in Q: \exists i\le l \quad \lda^{ij}\ne () \}$;
\item $e_j=d_j$ for all $j\in R$; and
\item $\mu^{ij}=\lda^{ij}$ for all $j\in R$, $i\in [1,l]$.
\end{enumerate}
Let $\pi:\GL_{\mbf n}(q)\ra \GL_{\mbf m}(q)$ be the usual projection onto the first $l$ components.
Observe that if $\mbf g\in \GL_{\mbf n}(q)$ is of type $\Cal T$, then 
$\pi(\mbf g)$ is of type $[R,\mbf e,\bs\mu]$.

\begin{lem}\label{typeproj} Let $\Cal T'=[Q,\mbf d,\bs \lda']$ be an $\mbf n$-type. Suppose that the projection of $\Cal T'$ onto the first $l$ components is of the form  
$\Cal T=[Q,\mbf d,\bs \lda]$ (so the two types have the same indexing set). Let $C$ be a class in $\GL_{\mbf m}(q)$ of type $\Cal T$. For a fixed $\Cal T'$,
the number of classes $C'$ of type $\Cal T'$ in $\GL_{\mbf n}(q)$ such that
$\pi(C')=C$ does not depend on $C$ or $q$.         
\end{lem}
\begin{proof} 
Let $\Cal P'=(Q,\mbf d,\bs\lda')$ and $\Cal P=(Q,\mbf d,\bs\lda)$ be 
corresponding pretypes. Clearly, $\Cal R_{\Cal P}(q)=\Cal R_{\Cal P'}(q)$. 
Moreover, if $F$ is a realisation of $\Cal P'$, then 
$\pi(C_{\Cal P'}(F))=C_{\Cal P}(F)$. By Lemma \ref{real}, the 
number of realisations $F$ such that $C_{\Cal P}(F)=C$ is $|\Aut(\Cal P)|$. 
By Lemma \ref{real} again, each class 
$C'$ of type $\Cal T'$ such that $\pi(C')=C$ is induced by 
$|\Aut(\Cal P')|$ of those realisations. It follows that
 the number of such classes $C'$
is $|\Aut(\Cal P)|/|\Aut(\Cal P')|$. 
This, clearly, does not depend on $C$ or $q$. 
\end{proof}

\begin{lem}\label{subproj} Let $H=H(q)\subseteq \GL_{\mbf n} (q)$ be a 
uniform family of subgroups. Assume that $\pi|_H:H\ra \pi(H)$ is injective 
for all $q$. Suppose also that, for any $\mbf g=(g_1,\ldots,g_n)\in H$, any 
polynomial appearing in $g_i$ for any $i>l$ also appears in $g_j$ for 
some $j\le l$. Then $\pi(H)\subseteq \GL_{\mbf m}(q)$ is a uniform family
of subgroups.
\end{lem}

\begin{proof} Let $\Cal T=[Q,\mbf d,\bs \mu]$ be an $\mbf m$-type. 
Let $C\subseteq \GL_{\mbf m}(q)$ be a class of type $\Cal T$. 
Since $\pi|_H$ is injective,
$|\pi(H)\cap C|=|H \cap \pi^{-1}(C)|$. Now $\pi^{-1}(C)$ is a union of
 conjugacy classes $C'$ of $\GL_{\mbf n}(q)$. If $\Cal T'$ is 
an $\mbf n$-type, let $D_{\Cal T'}(C)$ be the set of all classes 
$C'\subseteq \GL_{\mbf n}(q)$ of type $\Cal T'$ such that $\pi(C')=C$. 
Then
\begin{equation}\label{uniB1}
|\pi(H)\cap C| = |H\cap\pi^{-1}(C)| =  
\sum_{\Cal T'} \sum_{C'\in D_{\Cal T'}(C)} |C'\cap H|
\end{equation} 
where the first sum is over all $\mbf n$-types $\Cal T'$ whose projection onto the first $l$ components is $\Cal T$. Since $H$ is a uniform family, 
for a fixed $\Cal T'$, $|C'\cap H|$ does not 
depend on the choice of $C'\in D_{\Cal T'}(C)$ and is a 
polynomial function of $q$. 
Moreover, it follows from the hypothesis that
$|C'\cap H|=0$ (where $C'\in D_{\Cal T'}(C)$) 
unless $\Cal T'$ is of the form $[Q,\mbf d, \bs \mu']$ (for some $\bs \mu'$), 
i.e. $\Cal T'$ has the same indexing set as $\Cal T$. 
Therefore, \eqref{uniB1} yields
\begin{equation}\label{uniB2}
|\pi(H)\cap C| = \sum_{\Cal T'} |D_{\Cal T'}(C)| |C' \cap H|
\end{equation} 
where the sum is over all $\mbf n$-types $\Cal T'$ of the form 
$[Q,\mbf d,\bs \mu']$ (for some $\bs\mu'$) whose projection on the first $l$
components is $\Cal T$ ($C'\in D_{\Cal T'}(C)$ is chosen arbitrarily). 
By Lemma \ref{typeproj}, $|D_{\Cal T'}(C)|$ depends only on $\Cal T'$, 
but not on $C$ or $q$. Since $H$ is a uniform family of subgroups, for a fixed $\Cal T$, $|\pi(H)\cap C|$ does not 
depend on $C$ and is polynomial in $q$. \end{proof}

\begin{lem}\label{uniprod} Let $\mbf n$ and $\mbf k$ be 
tuples of positive integers.
Suppose that $H_1=H_1(q)$ and $H_2=H_2 (q)$ are
uniform families of subgroups in $\GL_{\mbf n}(q)$ and $\GL_{\mbf k}(q)$ respectively. Then $H_1\times H_2$ is a uniform family in 
$\GL_{\mbf n, \mbf k} (q)$. 
\end{lem}
\begin{proof} Let $\Cal T$ be an $(\mbf n,\mbf k)$-type. 
Let $C\subseteq \GL_{\mbf n,\mbf k}(q)$ be a class of type
$\Cal T$. Let $C_1$ and $C_2$ be the projections of $C$ into $\GL_{\mbf n}(q)$
and $\GL_{\mbf k}(q)$. The types of $C_1$ and $C_2$ are uniquely determined
by $\Cal T$. Since
$$
|(H_1 \times H_2) \cap C| = |H_1 \cap C_1||H_2\cap C_2|,
$$
the result follows.
\end{proof}

\section{A theorem of Higman}\label{higthm} 

Suppose $H$ is a PORC family of subgroups.
An algebraic family of groups given by a morphism 
$\Gamma:\GL_{\mbf n}(\mQ)\ra \GL_m (\mQ)$ defines, 
for each prime power $q$, an action of $H(q)$ on the vector space 
$\mF_q^m$ over the field $\mF_q$. We are interested in the behaviour of $\gamma(H(q),\mF_q^m)$, the number of 
orbits of this action, as $q$ varies.
         
\begin{thm} \label{Hig} Let $\Gamma$ be an algebraic family of groups, mapping $\GL_{\mbf n}(K)$ into $\GL_{m}(K)$. Let $H(q)$ be a 
PORC family of subgroups of $\GL_{\mbf n}(q)$, where $q$ runs through all primes or all prime powers. Then there exists a finite set $E$ of primes such that
 $\gamma(H(q),\mF_q^m)$ is PORC function on the set of all primes (respectively, prime powers) $q$ that are not powers of primes from $E$. 
\end{thm}

\begin{proof}[Sketch proof] This is a generalisation of~\cite[Theorem 1.2.1]{Higman}, 
which is (almost) equivalent to the present theorem in the case when 
$H(q)=\GL_{\mbf n}(q)$. The proof is essentially the same. 
The only necessary changes to the proof of~\cite{Higman} are as follows:
\begin{enumerate}[(i)]
\item \cite[Lemma 3.2]{Higman} should be replaced by the hypothesis that $H(q)$ is a PORC family of subgroups; 
\item Minor adjustments are needed on account of the fact that the notion of a type is different from that of an H-type. (The proof works in the same way, but the result is more general if the notion of a type defined above is used.)
\end{enumerate}

A reader who understands the proof of~\cite{Higman} will have no difficulty in making these changes.
For a detailed translation of Higman's proof, 
see~\cite[Appendix A]{thesis}. 

We sketch the idea of the proof briefly.
The well-known orbit-counting formula (see e.g.~\cite{Burnside}) yields the 
expression 
$$
\gamma(H(q),\mF_q^m)= |H(q)|^{-1}\sum_{\mbf g\in H(q)} \Fix(\Gamma(\mbf g)) 
$$
where $\Fix(\Gamma(\mbf g))$ is the number of points of $\mF_q^m$ fixed by 
$\Gamma(\mbf g)$. 
By the hypothesis and Lemma~\ref{numcl}, $|H(q)|$ is a PORC function of $q$. 
Hence, by Lemma \ref{PORClem}, it is enough to show that 
$$
 \sum_{g\in H(q)} \Fix(\Gamma(\mbf g)) 
$$
is PORC. Of course, $\Fix(\Gamma(\mbf g))=q^t$ where $t$ is the multiplicity of the eigenvalue $1$ in 
$\Gamma(\mbf g)$. Lemma 3.1 of~\cite{Higman} shows that, if the type $\Cal T$ 
of $\mbf g$ is fixed, the eigenvalues of
$\Gamma(\mbf g)$ over the algebraic closure $\ol{\mF}_q$ of $\mF_q$, with multiplicity, can be expressed as (fixed) monomials of the eigenvalues of $\mbf g$ over $\ol{\mF}_q$ as long as $q$ is a power of a prime outside a fixed finite set $E$. A result of Higman  
(\cite[Theorem 2.2.2]{Higman}) shows that the number of solutions of a system of simultaneous 
`monomial' equations and inequalities over $\mF_{q^k}$ (where $k\in \mN$) is a PORC function of $q$.
Using these two results, one can deduce that, for a fixed $t\ge 0$, the number of ways to choose a conjugacy class $C$ of type $\Cal T$ so that, if $\mbf g\in C$, $\Gamma(\mbf g)$ has eigenvalue $1$ with multiplicity $t$, is a PORC function of $q$. It remains to note that, by the hypothesis, if $C$ is such a conjugacy class, then
$|C\cap H(q)|$ is a PORC function of $q$, and to sum over all possible types $\Cal T$.
\end{proof}

\begin{remark} If $q$ runs through all primes, then by Lemma~\ref{PORClem}
 \eqref{finporc} we may assume that $E=\varnothing$.
\end{remark}

\section{Hall polynomials}\label{Hallpoly}

Let $\fr{o}$ be a (commutative) discrete valuation ring, and let $\fr{p}$ be its maximal ideal. Assume that $\fr{o}/\fr{p}$ is a finite field isomorphic to 
$\mF_q$. For background on discrete valuation rings, see 
\cite[Chapter 9, \S 1]{AM}. Typical examples of such rings include 
$\mF_q [[X]]$ and $\mZ_p$. 

For each partition $\lda=(\lda_1,\ldots,\lda_s)$, there is a corresponding $\fr o$-module  
$$
M_{\lda}=M_{\fr o,\lda}:=\frac{\fr o}{\fr p^{\lda_1}} \oplus 
\frac{\fr o}{\fr p^{\lda_2}} \oplus \cdots \oplus 
\frac{\fr o}{\fr p^{\lda_s}}.  
$$
Since $\fr o$ is a principal ideal domain, every finite $\fr o$-module 
is of the form $M_{\lda}$ for some partition $\lda$.
We make use the following well-known results. 
\begin{thm}\label{Hall}\emph{\cite[II.4]{Macdonald}}  Let $\lambda,\mu,\nu$ be (fixed) partitions. Then the number of submodules $N$ of $M_{\lda}$ such that 
$N \simeq M_{\mu}$ and $M_{\lda}/N \simeq M_{\nu}$ depends only on $q$ (but not on $\fr{o}$) and is a polynomial function of $q$.
\end{thm} 
A straightforward induction yields the following.
\begin{cor}\label{Hall:ind}
Let $\lambda,\mu^1,\mu^2,\ldots,\mu^r$ be partitions. The number of chains 
$$
0=N_{0}\le N_1 \le N_2 \le \cdots \le N_{r} = M_{\fr o, \lda} 
$$
of $\fr o$-submodules of $M_{\lda}$ such that $N_i/N_{i-1}\simeq M_{\mu^i}$ for $i=1,\ldots,r$ 
depends only on $q$ (but not on $\fr o$) and is a polynomial function of $q$.
\end{cor}
\begin{prop}\label{auto}\emph{\cite[II.(1.6)]{Macdonald}} The number of automorphisms of a finite $\fr{o}$-module $M_{\lda}$ of a fixed elementary divisor type $\lambda$ depends only on $q$ (but not on $\fr o$) and is a polynomial function of $q$.
\end{prop}
The polynomials appearing in Theorem \ref{Hall} are known as \emph{Hall polynomials}. Let $a_{\lda}(q)$ be the polynomials of Proposition \ref{auto}. That is,
$$
a_{\lda} (q)= |\Aut (M_{\fr o,\lda})|
$$
where $|\fr o/\fr p|=q$.

We shall apply these results in the following context 
(see \cite[IV.2, page 271]{Macdonald}). 
If $f\in \mF_q [X]$ is an irreducible polynomial, let $\mF_q [X]_{(f)}$ be the localisation of 
$\mF_q [X]$ at the prime ideal $(f)$; that is, $\mF_q [X]_{(f)}$ is the ring of fractions $u/v$ where 
$u,v\in \mF_q [X]$ and $v\notin (f)$. Then $\mF_q [X]_{(f)}$ is a discrete valuation ring with maximal ideal $(f)$ and residue field $\mF_q [X]/(f)$ that consists of $q^{\deg f}$ elements.  

\begin{lem}\label{autogen}\emph{\cite[Lemma 3.2]{Higman}} 
Let $n\in \mN$. Let $\Cal T$ be an $(n)$-type.
Then there exists a polynomial $F_{\Cal T}(q)$ such that, for every prime power $q$ and every class $C$ of $\GL_n (q)$ of type $\Cal{T}$, 
$|C|=F_{\Cal T} (q)$.  
\end{lem}
\begin{proof} Let $g$ be an element of a class $C$
of type $\Cal T$ in $G:=\GL_n (q)$. By the orbit-stabiliser formula,
\begin{equation}\label{sizeC}
|C|= \frac{|G|}{|C_G (g)|}.
\end{equation}
However, the elements of $G$ commuting with $g$ are precisely the automorphisms of the $\mF_q [X]$-module given by the action of $g$ on $\mF_q$.
This module is of the form 
$$
M_{\lda^1}(f_1)\oplus\cdots\oplus M_{\lda^r}(f_r),
$$
where $\lda^i$ are partitions and 
$f_i\in \mF_q [X]$ are irreducible polynomials. Since $C$ is of type $\Cal T$,
$$
\Cal T=[[1,r],(\deg f_i)_{i\in [1,r]}, (\lda^i)_{i\in [1,r]}].$$
The group of automorphisms of this module is the direct product of 
$\Aut(M_{\lda^i} (f_i))$, $i=1,\ldots,r$.  
Note that automorphisms of $M_{\lda^i}(f_i)$ as an $\mF_q [X]$-module are the same as its automorphisms as an $\mF_q [X]_{(f_i)}$-module. Hence, by 
Proposition \ref{auto}, 
$$
|\Aut(M_{\lda^i}(f_i))|=a_{\lda_i}(q^{\deg f_i}).
$$
By \eqref{sizeC},
$$
|C|=\frac{|\GL_n (q)|}{\prod_{i=1}^r a_{\lda_i} (q^{\deg f_i})}.
$$
Clearly, this depends only on $\Cal T$ and $q$ (but not on $C$). It is well known that $|\GL_n (q)|$ is polynomial in $q$: in fact, $|\GL_n (q)|=a_{(1^n)}(q)$. Hence, by Lemma \ref{PORClem} (\ref{quot}), $|C|$ is polynomial in $q$.
\end{proof}

Let $V=\mF_q^n$, the standard $n$-dimensional vector space over $\mF_q$. Fix a 
tuple $\mbf d=(d_1,\ldots,d_l)$ of positive integers 
 such that $d_1+\cdots+d_l=n$. Write 
$\mbf a = (n,d_1,\ldots,d_l)$. Let $S$ be the set of all tuples
$\mbf U=(U_1,\ldots,U_l)$ such that
\begin{equation}\label{chain1}
0=U_0\le U_1\le U_2\le \dots\le U_l=V
\end{equation}
is a chain of subspaces (a flag) of $V$ and $\dim(U_i/U_{i-1})=d_i$. The group 
$\GL(V)$ acts on the set $S$.
Recall that $\scr P(V;U_0,U_1,\ldots,U_l)$ is the (parabolic) subgroup of $\GL(V)$ that consists of 
the elements stabilising $\mbf U=(U_1,\ldots,U_l)$.
For each $\mbf U\in S$, there is a natural embedding
$$
\tau=\tau_{\mbf U}:\scr P(V;U_0,U_1,\ldots,U_l) \ra 
\GL(V) \times \prod_{i=1}^{l} \GL(U_i/U_{i-1})
\simeq \GL_{\mbf a}(q).
$$
Let $\Cal J (\mbf U)=\im \tau$. We shall identify 
$$
\GL(V) \times\prod_{i=1}^{l} \GL(U_i/U_{i-1})
$$ 
with $\GL_{\mbf a}(q)$, so $\Cal J (\mbf U)$ may be seen as a subgroup of 
$\GL_{\mbf a}(q)$. Let $\Cal J_{\mbf d}(q)=\Cal J (\mbf U)$ for an 
arbitrarily chosen $\mbf U \in S$. (Thus, $\Cal J_{\mbf d}(q)$ is determined 
only up to conjugacy, but that is all we shall need.) 
Our aim is to show that $\Cal J_{\mbf d}(q)$ is a uniform family of subgroups. 
Let 
$\pi:\GL_{\mbf a}(q) \ra \GL_{n}(q)\simeq \GL(V)$ 
be the natural projection.

\begin{prop} \label{flag} Let $V=\mF_q^n$. 
Fix an $\mbf a$-type $\Cal T$. Let $C$ be a class in $\GL_{\mbf a}(q)$ of type $\Cal T$.
Let $g\in \pi(C) \subseteq \GL(V)$. 
Then the number of flags $\mbf U\in S$ such that $g(\mbf U)=\mbf U$ 
and $\tau_{\mbf U}(g)\in C$ 
depends only on $\mathcal{T}$ and $q$, but not on $C$ or $g$, 
and is a polynomial function of $q$.
\end{prop}
\begin{proof} The map $g\in \GL(V)$ makes $V$ 
an $\mF_q [X]$-module in the usual way. Then 
$V$ can be decomposed as a direct sum of modules:
$$
V=V_1\oplus V_2\oplus\cdots\oplus V_s,
$$
where $V_i\simeq M_{\kappa^i}(f_i)$ ($i\in [1,s]$) for some partitions 
$\kappa^1,\ldots,\kappa^s$ and distinct irreducible polynomials 
$f_1,\ldots,f_s\in \mF_q [X]$. 
By a well-known result of linear algebra, the projection map from $V$ 
onto each $V_i$ is polynomial in $g$. 
Hence, any subspace $W$ of $V$ satisfying $g(W)=W$ decomposes as 
$$
 W=(W\cap V_1) \oplus (W\cap V_2)\oplus \cdots \oplus (W\cap V_s). 
$$
Suppose $\mbf U\in S$ is chosen so that $g(\mbf U)=\mbf U$ and 
$\tau_{\mbf U}(g)\in C$. 
Let $U_{ij}=U_i\cap V_j$ 
($1\le i \le l$, $1\le j \le s$).
Let $\mu^{ij}$ be the partition satisfying 
$U_{ij}/U_{i-1,j} \simeq M_{\mu^{ij}}(f_j)$. 
Since $\tau_{\mbf U}(g)$ is of type $\Cal T$, we conclude that
$$
\Cal T = [[1,s],(\deg f_i)_{i\in [1,s]}, (\lda^{ij})_{i\in [1,l+1], j\in [1,s]} ]
$$ 
where $\lda^{1j}=\kappa^j$ and $\lda^{i+1,j}=\mu^{ij}$ for all $i,j$ (recall that the partitions 
$\lda^{1j}$ correspond to the component $\GL(V)$ and $\lda^{i+1,j}$ 
correspond to 
$\GL(U_i/U_{i-1})$). 

Choosing a flag $\mbf U\in S$ such that $g(U_i)=U_i$ for all $i$ and
 $\tau_{\mbf U}(g)\in C$ amounts to choosing, for each $j\in [1,s]$, a chain
\begin{equation}\label{chain2}
0=U_{0j} \le U_{1j} \le \cdots \le U_{lj} = V_j
\end{equation}
of submodules of the $\mF_q [X]_{(f_j)}$-module $V_j$ satisfying 
$U_{ij}/U_{i-1,j}\simeq M_{\mu^{ij}}$. 
By Corollary \ref{Hall:ind}, there exists a polynomial $R_j$ (depending only on $\kappa^j$ and 
$\mu^{ij}$) such that the number of ways of choosing such a chain is 
$R_j (q^{\deg f_j})$ for all prime powers $q$. 
Hence, the number flags $\mbf U\in S$ 
preserved by $g$ such that $\tau_{\mbf U}(g)\in C$ is 
$$
\prod_{j=1}^s R_j (q^{\deg f_j}).
$$ 
This is a polynomial function of $q$, uniquely determined by $\Cal T$. \end{proof} 

\begin{prop} \label{Hall:gen} $\Cal J_{\mbf d}(q)$ is a uniform family of 
subgroups in $\GL_{\mbf a}(q)$.
\end{prop}
\begin{proof}
Fix an $\mbf a$-type $\mathcal{T}$. If $C$ is a class in 
$\GL_{\mbf a}(q)$ of type 
$\Cal T$, let $\alpha (q,C)=|\Cal J_{\mbf d}(q) \cap C|$. 
We are to show that $\alpha(q,C)$ does not depend on $C$ and is a 
polynomial function of $q$.  
Let
$$
T=\{ (\mbf U,g) \in S\times \GL(V): g(\mbf U)=\mbf U \text{ and }
\tau_{\mbf U}(g)\in C \}.
$$ 
We shall count $|T|$ in two ways.   For any $\mbf U\in S$, the number of $g\in \GL(V)$ such that $(\mbf U,g)\in T$ does not depend on the choice of 
$\mbf U$. (This follows from the fact that $\GL(V)$ acts transitively on $S$.) 
Hence,
$$
|T|=\alpha(q,C)|S|.
$$
On the other hand, choose $g\in \pi(C)$, and let $\beta(q,C)$ be the number of flags $\mbf U\in S$ such that $g(\mbf U)=\mbf U$ and $\tau_{\mbf U} (g) \in C$.
 By Proposition \ref{flag}, 
$\beta(q,C)$ does not depend on $C$ and is polynomial in $q$. Thus, $|T|=|\pi(C)|\beta(q,C)$.
It follows that
$$
\alpha(q,C)=\frac{|\pi(C)|\beta(q,C)}{|S|}.
$$ 
$\pi(C)$ is a conjugacy class in $\GL(V)$, and its type is uniquely determined by $\Cal T$. 
Hence, by Lemma \ref{autogen}, $|\pi(C)|$ does not depend on $C$ 
and is polynomial in $q$. It is well known that $|S|$ is polynomial in $q$. 
The result now follows from Lemma \ref{PORClem} (\ref{quot}). \end{proof}

\section{Automorphisms of modules}\label{autsection}

As in the previous section, let $\fr o$ be a discrete valuation ring with a maximal ideal $\fr p$ and a finite residue field. We shall identify $\fr o/\fr p$ with $\mF_q$, where $q=|\fr o/\fr p|$. Let $t$ be a uniformizer of $\fr o$, i.e. an element of $\fr p\setminus\fr p^2$.
Let $\lda=(\lda_1,\ldots,\lda_s)$ be a partition. In this section, 
we investigate the automorphism group of the module $B:=M_{\fr o,\lda}$ 
in some detail. Let
$$
B[\fr p]=B[t]=\{x\in B: tx=0 \}.
$$
Then $B[\fr p]$ is a vector space over $\mF_q$, as is $B/\fr p B$. There is a natural map
$$
\beta: \Aut(B) \ra \GL(B/ \fr p B) \times \GL(B[\fr p]),
$$
defined by the formula 
$$
\beta(h)(x+\fr p B,y)=(hx+ \fr pB, hy),
$$  
where $h\in \Aut(B)$, $x\in B$, $y\in B[\fr p]$.

Our aim is to describe $\im\beta$ and, in particular, to prove that $\im\beta$ is a uniform family of subgroups of 
$\GL(B/\fr pB)\times \GL(B[\fr p])\simeq\GL_{s,s}(q)$. We can represent $\lda$ as 
$$
(\underbrace{\mu_1,\ldots,\mu_1}_{u_1},\ldots,\underbrace{\mu_r,\ldots,\mu_r}
_{u_r})
$$
where $\mu_1>\cdots>\mu_r$. Then
$$
B=\fr o e_{11} \oplus \fr o e_{12} \oplus \cdots \oplus \fr o e_{1,u_1} 
\oplus \cdots \oplus \fr o e_{r1} \oplus \cdots \oplus \fr o e_{r,u_r} 
$$
for some $e_{ij}\in B$ satisfying $\Ann_{\fr o}(e_{ij})=\fr p^{\mu_i}$. If 
$h\in\Aut (B)$, $h$ can be represented by a (not unique) matrix 
$X\in \GL_s (\fr o)$ so that, for any $a_{11},\ldots,a_{r,u_r}\in \fr o$,
\begin{equation}\label{hmat}
h \left(\sum_{i=1}^r \sum_{j=1}^{u_i} a_{ij} e_{ij} \right)= 
(e_1,\ldots,e_{r,u_r}) X 
\begin{pmatrix}
a_{11} \\ \vdots \\ a_{r,u_r} \\
\end{pmatrix}
\end{equation}
It is easy to check that a matrix $X\in \GL_s (\fr o)$ gives rise to an automorphism of $B$ in this way if and only if 
\begin{equation}\label{Xmat}
X= \begin{pmatrix}
X_{11} & t^{\mu_1-\mu_2} X_{12} & \ldots & t^{\mu_1-\mu_r} X_{1r} \\
 X_{21} & X_{22} & \ldots & t^{\mu_2-\mu_r} X_{2r} \\
\vdots & \vdots & \ddots & \vdots \\  
 X_{r1} & X_{r2} & \ldots & X_{rr} \\
\end{pmatrix}
\end{equation}
where each $X_{ij}$ is an arbitrary $u_i\times u_j$ matrix over $\fr o$. 

Clearly, $\{\bar{e}_{11},\ldots,\bar{e}_{r,u_r}\}$, where 
$\bar{e}_{ij}=e_{ij}+\fr p B$, is a basis of $B/\fr p B$. Let 
$f_{ij}=t^{\mu_i-1}e_{ij}$ for all $i,j$. Then 
$\{ f_{11},\ldots,f_{r,u_r} \}$ is a basis of $B[\fr p]$. Using these bases, we shall identify $\GL(B/\fr p B)\times \GL(B[\fr p])$ 
with $\GL_{s,s} (q)$. Thus, an element 
$(Y,Z)\in \GL(B/\fr p B)\times \GL(B[\fr p])$ is given by a pair of 
invertible matrices of the form 
\begin{eqnarray} Y & = & 
\begin{pmatrix}
 Y_{11} & Y_{12} & \ldots & Y_{1r} \\

\vdots & \vdots & \ddots & \vdots \\ 
Y_{r1} & Y_{r2} & \ldots & Y_{rr} \\ 
\end{pmatrix}, \label{Ymat} \\ 
Z & = & \begin{pmatrix}
 Z_{11} & Z_{12} & \ldots & Z_{1r} \\
\vdots & \vdots & \ddots & \vdots \\ 
Z_{r1} & Z_{r2} & \ldots & Z_{rr} \\ 
\end{pmatrix} \label{Zmat}
\end{eqnarray}      
where $Y_{ij}$ and $Z_{ij}$ are $u_i\times u_j$ matrices over $\mF_q$. The next result easily follows from the description of automorphisms of $B$ given in the previous paragraph.
\begin{lem}\label{autim} Let $Y\in \GL(B/\fr p B)$ and $Z \in \GL(B[\fr p])$ be as in 
\eqref{Ymat}, \eqref{Zmat}. Then $(Y,Z)\in \im \beta$ if and only if the following hold:
\begin{enumerate}[(i)]
\item $Y_{ij}=0$ whenever $i<j$,
\item $Z_{ij}=0$ whenever $i>j$, and
\item $Y_{ii}=Z_{ii}$ for all $i\in [1,r]$.
\end{enumerate} 
\end{lem}

Lemma \ref{autim} describes $\im\beta$ as a subgroup of $\GL_{s,s}(q)$ (up to conjugacy). This description only depends on 
$\mbf u=(u_1,\ldots,u_r)$ and $q$: 
it does not depend on $\fr o$. Let 
$H=H^{\mbf u}(q)=\im\beta \subseteq \GL_{s,s}(q)$. 
Let $\mbf a=(s,u_1,\ldots,u_r)$, $\mbf b=(s,s,u_1,\ldots,u_r)$.
Let $\sigma:H\hra \GL_{\mbf b}(q)$ be the injective homomorphism given by
$$
(Y,Z) \mapsto (Y,Z,Y_{11},Y_{22},\ldots,Y_{rr}),
$$
where $Y_{ii}$ are as in \eqref{Ymat}. Let 
$\pi_1:\GL_{\mbf b}(q)\ra \GL_{\mbf a}(q)$ be the projection onto the components $1,3,4,\ldots,r+2$. Let $\pi_2:\GL_{\mbf b}(q)\ra \GL_{\mbf a}(q)$ be the projection onto the components $2,3,4,\ldots,r+2$. Also, let 
$\pi_0:\GL_{\mbf b}(q)\ra \GL_{\mbf u}(q)$ be the projection onto the components $3,4,\ldots,r+2$. 
\begin{lem}\label{sigH} $\sigma(H)$ is a uniform family of subgroups in 
$\GL_{\mbf b}(q)$.
\end{lem}
\begin{proof} Let $\Cal T$ be a $\mbf b$-type. Let $C$ be a class in 
$\GL_{\mbf b}(q)$ of type $\Cal T$. Let $H_1=\pi_1(\sigma(H))$, 
$H_2=\pi_2 (\sigma(H))$.
By Lemma \ref{autim}, 
$(Y,Z,T_1,\ldots,T_r)\in \sigma (H)$ if and only if 
$(Y,T_1,\ldots,T_r)\in H_1$ and $(Z,T_1,\ldots,T_r)\in H_2$.
Let $(T_1,\ldots,T_r)\in \pi_0 (C)$. Then there are 
$$
\frac{|\pi_1 (C) \cap H_1|}{|\pi_0 (C)|}
$$
matrices $Y\in \GL_s (q)$ such that 
$(Y,T_1,\ldots,T_r)\in \pi_1(C)\cap H_1$.
This follows from the fact that 
the number of such matrices $Y$ does not depend on the choice of 
$(T_1,\ldots,T_r)$. Similarly, there are 
$$
\frac{|\pi_2 (C) \cap H_2|}{|\pi_0 (C)|}
$$
matrices $Z\in \GL_s (q)$ such that 
$(Z,T_1,\ldots,T_r)\in \pi_2(C)\cap H_2$.
Therefore,
\begin{eqnarray*}
|\sigma(H)\cap C|& = & 
|\pi_0 (C)| \cdot \frac{|\pi_1 (C) \cap H_1|}{|\pi_0 (C)|} \cdot
\frac{|\pi_2 (C) \cap H_2|}{|\pi_0 (C)|} \\ 
& = & \frac{|\pi_1 (C) \cap H_1||\pi_2 (C) \cap H_2|}{|\pi_0 (C)|}.
\end{eqnarray*}
By Lemma~\ref{autim}, $H_2$ is conjugate in $\GL_{\mbf a}(q)$ to the subgroup 
$\Cal J_{\mbf u}(q)$, and $H_1$ is conjugate to 
$\Cal J_{(u_r,\ldots,u_2,u_1)}(q)$. Also, the types of the classes $\pi_1(C)$,
$\pi_2(C)$ and $\pi_0 (C)$ are uniquely determined by $\Cal T$. 
By Proposition \ref{Hall:gen}, for $i=1,2$, $|\pi_i (C) \cap H_i|$ depends only on $\Cal T$ and $q$ (but not on $C$) and is a polynomial function of $q$. The same holds for $|\pi_0 (C)|$ by Lemma \ref{autogen}. 
By Lemma \ref{PORClem} (\ref{quot}), the result follows. \end{proof}

\begin{prop}\label{uniB} 
$H=H^{\mbf u}(q)$ is a uniform family of subgroups in 
$\GL_{s,s}(q)$.
\end{prop}
\begin{proof} Let $\pi:\GL_{\mbf b} (q) \ra \GL_{s,s}(q)$ be the projection 
onto the first two components. Then $\pi|_{\sigma(H)}$ is injective and
$\pi(\sigma(H))=H$. Moreover, 
for any $(Y,Z,Y_{11},\ldots,Y_{rr})\in \sigma(H)$,            
any irreducible polynomial appearing in $Y_{ii}$ for any $i$ also appears in 
$Y$. Thus, by Lemma \ref{subproj}, $\pi(\sigma(H))=H$ is a uniform family of
subgroups.
 \end{proof}

\section{An expression for the functor Ext}\label{Extsection}

Let $R$ be an integral domain (with identity). Let $A$ and $B$ be $R$-modules. The $R$-module 
$\Ext(A,B)=\Ext_R^1 (A,B)$ is defined as follows (see \cite[III.2]{HS}). Let
\begin{equation}\label{projrep}
0 \lra Q \stackrel{\mu}{\lra} P \stackrel{\epsilon}{\lra} A \lra 0
\end{equation} 
be an exact sequence of $R$-modules with $P$ projective. 
The map $\mu$ induces a homomorphism $\mu^*:\Hom(P,B) \ra \Hom(Q,B)$, given by 
$\delta \mapsto \delta\circ\mu$. 
Then 
$$
\Ext(A,B):=\coker(\mu^*)=\frac{\Hom(Q,B)}{\im \mu^*}.  
$$
It is shown in \cite[III.2]{HS} that $\Ext(-,-)$ is a bifunctor from the category of $R$-modules to itself, contravariant in the first and covariant in the second variable. Moreover, a different choice of a sequence \eqref{projrep} gives rise to the same functor up to natural equivalence.

For the remainder of this section, assume that $R$ is a principal ideal domain. If $A$ and $B$ are finitely generated modules over $R$, then both $A$ and $B$ are direct sums of cyclic modules.
Thus, it is easy to determine $\Ext(A,B)$ up to isomorphism using the isomorphisms 
$\Ext(A\oplus A',B)\simeq \Ext(A,B) \oplus \Ext(A',B)$ and
 $\Ext(A,B\oplus B')\simeq \Ext(A,B) \oplus \Ext(A,B')$ 
(see \cite[III.4]{HS}). 
However, our aim is to describe $\Ext(A,B)$ up to \emph{natural} equivalence.

Let $A$ be an $R$-module. Let 
$$
T(A)=\{a\in A: \exists r\in R\setminus \{0\}\quad ra=0 \} 
$$
be the torsion submodule of $A$. $T$ is a covariant functor from the category of $R$-modules to itself.

\begin{lem}\label{torsion} Suppose $A$ is a principal ideal domain. 
For finitely generated $R$-modules 
$A$ and $B$, $\Ext(A,B)$ is naturally equivalent to $\Ext(T(A),B)$.
\end{lem}
\begin{proof} The natural exact sequence 
$$
0 \lra T(A) \lra A \lra A/T(A) \lra 0 
$$ 
induces a natural long exact $\Ext$-sequence (see \cite[IV.7]{HS})
$$
\Ext(A/T(A), B) \lra \Ext(A,B) \lra \Ext(T(A), B) \lra \Ext_R^2 (A/T(A), B).  
$$ 
However, since $A/T(A)$ is a finitely generated free 
module, $\Ext_R^n (A/T(A), B)=0$ 
for all $n\ge 1$ by \cite[IV, Proposition 7.2]{HS}. 
\end{proof}
 
\begin{lem}\label{homiso} Let $R$ be a principal ideal domain. Suppose 
$P$ is a finitely generated free $R$-module and $B$ is a finitely generated $R$-module. 
Then the map $\tau: \Hom(P,R)\otimes B \ra \Hom(P,B)$ given by 
$$
\tau(f\otimes b)(a)=f(a)b \quad \forall f\in \Hom(P,R), b\in B, a\in P 
$$
is a natural equivalence of bifunctors.
\end{lem}
\begin{proof} It is easy to see that $\tau$ is well defined and is a natural transformation. 
 Let $P=Re_1\oplus \cdots \oplus Re_k$, and let   
$B=Rf_1\oplus \cdots \oplus Rf_m$. Then $\Ann_R (e_i)=0$ and $\Ann(f_j)=I_j$ 
where $I_j$ is an ideal of $R$. Let $e^i\in \Hom(P,R)$ be the function given by 
$e^i (e_j)=\delta_{ij}$. Let $h^i_j\in\Hom(P,B)$ be the homomorphism given by $e_i\mapsto f_j$,
$e_l \mapsto 0$ for $l\ne i$. 
Then
\begin{eqnarray*}
\Hom(P,R)\otimes B & = & \bigoplus_{i=1}^k \bigoplus_{j=1}^m R (e^i\otimes f_j)
 \text{ and }\\
\Hom(P,B)  & = & \bigoplus_{i=1}^k \bigoplus_{j=1}^m R h^i_j.
\end{eqnarray*}
Since $\tau(e^i\otimes f_j)=h^i_j$ and $\Ann(e^i \otimes f_j)=\Ann(h^i_j)=I_j$, $\tau$ is an 
isomorphism. \end{proof}

Let $K$ be the field of fractions of $R$, and let $L=K/R$. If $A$ is an $R$-module, let 
$\widehat{A}=\Hom_R (T(A),L)$.

\begin{lem}\label{Extlem}\emph{\cite[VII, Proposition 2.3]{CE}} Suppose $R$ is a principal ideal domain. For a finitely generated torsion $R$-module $A$, $\Ext(A,R)$ is naturally 
equivalent to $\widehat{A}$.
\end{lem}

\begin{proof} 
We repeat the proof in \cite{CE} for convenience. The exact sequence 
$$
0 \lra R \lra K \lra L \lra 0
$$  
induces a natural Hom-Ext exact sequence 
$$
\Hom(A,K) \lra \Hom(A, L) \lra \Ext(A, R) \lra \Ext(A,K)
$$
(see \cite[III, Theorem 5.2]{HS}). Since $A$ is a torsion module, $\Hom(A,K)=0$. Since $K$ is an injective $R$-module, by \cite[III, Proposition 2.6]{HS}, $\Ext(A,K)=0$. The result follows.
\end{proof}  

\begin{thm}\label{Ext} Suppose $R$ is a principal ideal domain. For finitely generated $R$-modules $A$ and $B$, $\Ext(A,B)$ is naturally equivalent to $\widehat{A}\otimes B$.
\end{thm}

\begin{proof} By Lemma~\ref{torsion}, we may assume that $A$ is a torsion module. Then, by 
Lemma~\ref{Extlem}, $\Ext(A,R)$ is naturally equivalent to $\widehat{A}$. It remains to show 
that $\Ext(A,B)$ is naturally equivalent to $\Ext(A,R)\otimes B$. 

To see this, consider an 
exact sequence \eqref{projrep} with $P$ a finitely generated free $R$-module. 
Then $Q$ is also free. Let $\alpha: \Hom(P,R) \ra \Hom(Q,R)$ and $\mu^*: \Hom(P,B)\ra \Hom(Q,B)$ be the maps induced by $\mu$.
Lemma~\ref{homiso} yields natural equivalences 
$\tau_1: \Hom(P,R)\otimes B \ra \Hom(P,B)$ and $\tau_2: \Hom(Q,R)\otimes B \ra \Hom(Q,B)$. 
It is 
easy to check that the diagram
$$\xymatrix{
\Hom(P,R) \otimes B \ar[r]^{\alpha\otimes\id_B} \ar[d]^{\tau_1} & 
\Hom (Q,R) \otimes B \ar[d]^{\tau_2} \\
\Hom(P,B) \ar[r]^{\mu^*} & \Hom(Q,B) 
}
$$   
is commutative. Thus, $\Ext(A,B)=\coker \mu^*$ is naturally equivalent to 
$$
\frac{\Hom(Q,R)\otimes B}{(\im\alpha) \otimes B}.
$$
This is naturally equivalent to
$(\coker \alpha)\otimes B=\Ext(A,R)\otimes B$ because the functor $-\otimes B$ is right exact.
\end{proof}

\section{Conclusion of the proof} \label{finproof}

Let $p$ be a prime number.
Let $A$ and $B$ be finite abelian $p$-groups. We can view $A$ and $B$ as modules over the discrete valuation ring $\mZ_p$. Moreover, $A$ and $B$ may be seen as
Lie algebras over $\mZ_p$ with the zero Lie bracket operation. Call an exact sequence 
\begin{equation}\label{eq:extension}
0\lra B\stackrel{\iota}{\lra} E\stackrel{\pi}{\lra} A\lra 0 
\end{equation}
of $\mZ_p$-Lie algebras a \emph{central extension} 
of $A$ by $B$ if $\iota(B)$  is central in $E$, i.e. $[E,\iota(B)]=0$.

Another central extension 
$$0\lra B\lra E'\lra A\lra 0$$ 
is said to be \emph{equivalent} to \eqref{eq:extension} 
if there is a Lie algebra isomorphism $E\ra E'$ making the diagram
\begin{displaymath}
\xymatrix{ &  &  E \ar[dr] \ar[dd] &  &  \\
0 \ar[r] & B \ar[ur] \ar[dr] & & A \ar[r] & 0  \\
& & E' \ar[ur] & &
}
\end{displaymath}
commute. Let $\Cal L(A,B)$ be the set of all equivalence classes 
of central extensions of $A$ by $B$. This definition is very similar to 
the definition of 
$\Ext(A,B)$ (as a set rather than an abelian group) through exact sequences 
(see \cite[III.1]{HS}). 
The only difference is the Lie bracket operation; since the image of 
$B$ is central, that operation corresponds to a homomorphism 
from $\bigwedge^2 A$ into $B$. 

The group $\Aut(A)\times \Aut(B)$ acts on $\Cal L(A,B)$ as follows: if 
$g\in \Aut(A)$, $h\in\Aut(B)$, $(g,h)$ maps the equivalence class of the sequence \eqref{eq:extension} to the equivalence class of
$$
0\lra B \xrightarrow{\iota\circ h^{-1}} E \xrightarrow{g \circ \pi} 
A \lra 0.
$$

\begin{prop}\label{Lnat} If $A$ and $B$ are finite $\mZ_p$-modules, 
$\Cal L(A,B)$ is isomorphic, as an $\Aut(A)\times\Aut(B)$-set, to
$$
\Hom_{\mZ_p}\left(\bigwedge\nolimits^2 A, B\right) 
\oplus \left( \widehat{A}\otimes B \right).
$$
\end{prop}
\begin{proof} Consider an element $x\in \Cal L(A,B)$ given by \eqref{eq:extension}. Define
a homomorphism $y:\bigwedge^2 A\ra B$ by 
$$
a_1\wedge a_2 \mapsto \iota^{-1}([\tilde{a}_1,\tilde{a}_2])
$$
where
$a_1,a_2\in A$ and $\tilde{a}_i\in\pi^{-1}(a_i)$. 
Since $\iota(B)$ is central in $E$, this does not depend on the choices 
of $\tilde{a}_1$, $\tilde{a}_2$. We may view \eqref{eq:extension} as an exact 
sequence of $\mZ_p$-modules. By \cite[III.2]{HS}, equivalence classes of such
sequences are in a one-to-one correspondence with $\Ext(A,B)$. 
Let 
$z$ be the element of $\Ext_{\mZ_p}(A,B)$ corresponding to $x$. 
Clearly, $x\mapsto (y,z)$ defines 
a natural equivalence of $\Cal L(A,B)$ and 
$$
\Hom\left(\bigwedge\nolimits^2 A,B\right) \oplus \Ext(A,B).
$$ 
By Theorem \ref{Ext}, the result follows. \end{proof} 

Now let $\kappa$ and $\lda$ be partitions, and consider the $\mZ_p$-modules (hence, Lie algebras) $M_{\mZ_p,\kappa}=M_{\kappa}$ and $M_{\lda}$ in the roles of $A$ and $B$ respectively.
 Let $L_{\kappa,\lda}(p)$ be the set of 
isomorphism classes of pairs $(E,B)$ such that
\begin{enumerate}[(i)]
\item $E$ is a Lie algebra over $\mZ_p$;
\item $B$ is an ideal of $E$ isomorphic (as a Lie algebra) to $M_{\lda}$;
\item $E/B\simeq M_{\kappa}$; 
\item $[E,B]=0$.
\end{enumerate}
Naturally, an \emph{isomorphism} between two such pairs $(E,B)$ and $(E',B')$ is an isomorphism $f:E\ra E'$ of Lie algebras such that $f(B')=B$. 
Write $[(E,B)]$ for the isomorphism class of $(E,B)$. 
Each $[(E,B)] \in L_{\kappa,\lda} (p)$ gives rise to an exact sequence
\begin{equation}\label{exact1}
0\lra B \lra E \lra \frac{E}{B} \lra 0. 
\end{equation}
If we use arbitrarily chosen isomorphisms $B\ra M_{\kappa}$ and 
$E/B\ra M_{\lda}$, then \eqref{exact1} becomes 
\begin{equation}\label{exact2}
0 \lra M_{\lda} \lra E \lra M_{\kappa} \lra 0,
\end{equation}
giving rise to an element of  
$\Cal L(M_{\kappa},M_{\lda})$.
Clearly, the $\Aut(M_{\kappa})\times \Aut(M_{\lda})$-orbit of 
\eqref{exact2} does not depend on the choice of the isomorphisms. 
Moreover, it is easy to see that this 
establishes a one-to-one correspondence between 
$L_{\kappa,\lda}(p)$ and 
$\Aut(M_{\kappa})\times \Aut(M_{\lda})$-orbits on $\Cal L (M_{\kappa},M_{\lda})$.
This, together with Lemma \ref{Lnat}, yields the following result.

\begin{cor}\label{corr1} Let $\kappa$ and $\lda$ be partitions. The set 
$L_{\kappa,\lda}(p)$
is in a one-to-one correspondence with 
$\Aut(M_{\kappa})\times \Aut(M_{\lda})$-orbits on 
$$
\Hom\left(\bigwedge\nolimits^2 M_{\kappa},M_{\lda} \right) \oplus 
\left( \widehat{M}_{\kappa}\otimes M_{\lda} \right).
$$
\end{cor} 

Note that any finite Lie algebra $E$ of class at most 2 over 
$\mZ_p$ has a central ideal $B$ such that $[(E,B)]\in L_{\kappa,\lda}(p)$ 
for some partitions $\kappa$ and $\lda$: one can take $B=Z(E)$, for example. 
However, such a central ideal $B$ is, in general, not unique. 
  
For the rest of the section, we assume that $\kappa=(1^m)$ for some $m$, 
so $M_{\kappa}$ is an elementary abelian $p$-group, 
which may be identified with a vector space $V\simeq \mF_p^m$ over $\mF_p$. 
The $\mZ_p$-module $\widehat{M}_{\kappa}$ is then identified with $V^*$.  
By Lemma \ref{Lnat},
\begin{eqnarray}
\Cal L(V,M_{\lda}) & 
\simeq &
\Hom\left(\bigwedge\nolimits^2 V, M_{\lda} \right)
\oplus \left(V^* \otimes M_{\lda} \right) \nonumber\\
& \simeq &
\Hom\left(\bigwedge\nolimits^2 V, M_{\lda}[p] \right)
\oplus \left(V^* \otimes \frac{M_{\lda}}{p M_{\lda}} \right)\label{Lnat2}
\end{eqnarray}
where the isomorphisms are those of 
$\Aut(M_{\kappa})\times \Aut(M_{\lda})$-sets. 
For the second isomorphism, observe that
$V^*\otimes M_{\lda}$ is naturally isomorphic to 
$V^*\otimes (M_{\lda}/p M_{\lda})$ and that the image of any $\mZ_p$-homomorphism
from $\bigwedge^2 V$ into $M_{\lda}$ is contained in $M_{\lda}[p]$. 

Using this, one could apply Theorem 
\ref{Hig} to show that $|L_{\kappa,\lda}(p)|$ is a PORC function of $p$. However,
 our aim is to count only 
those elements $[(E,B)]$ of $L_{\kappa,\lda}(p)$ that satisfy $Z(E)=B$. 
We shall do this by an inductive argument at the end of this section. 
We need a more general construction for this purpose. 

Let $\mbf d=(d_1,\ldots,d_l)$ be a sequence of nonnegative integers such that
$d_1+\cdots+d_l=m$ and $d_i>0$ whenever $1\le i\le l-1$. Fix a flag
$$
0 = U_{0} < U_1 < U_2 < \cdots < U_{l-1} \le V
$$
in $V$ such that $\dim(U_i)-\dim(U_{i-1})=d_i$. (For convenience, let $U_l=V$.)

Let $F_{m,\mbf d,\lda}(p)$ be the set of all isomorphism 
classes of tuples $(E,B;W_1,\ldots,W_{l-1})$ such that 
\begin{enumerate}[(i)]
\item $E$ is a finite Lie algebra over $\mZ_p$ (of nilpotency class 2); 
\item  $B$ is a central ideal of $E$ isomorphic to $M_{\lda}$;
\item $E/B$ is an elementary abelian $p$-group of dimension $m$ over $\mF_p$ with the zero Lie bracket operation;
\item  $W_0,\ldots,W_{l-1}$ are ideals of $E$ and $B=W_0 \subset W_1 \subset \cdots \subset 
W_{l-1} \subseteq W_l:=E$;
\item  $\dim_{\mF_p} (W_i/B)-\dim_{\mF_p}(W_{i-1}/B)=d_i$ for $i=1,\ldots,l$;
\item \label{centcon} $W_{l-1}\subseteq Z(E)$.
\end{enumerate}

As before, let $[(E,B;W_1,\ldots,W_{l-1})]$ be the isomorphism 
class of $(E,F;W_1,\ldots,W_l)$. Each 
$[(E,B;W_1,\ldots,W_{l-1})]\in F_{m,\mbf d,\lda}(p)$ gives rise to an exact sequence
\begin{equation}\label{exact3}
0\lra B \lra E \lra \frac{E}{F} \lra 0.
\end{equation}
As above, we may use arbitrarily chosen isomorphisms $B\ra M_{\lda}$ and 
$$
(E/B;W_1/B,\ldots,W_{l-1}/B) \lra (V; U_1,\ldots,U_{l-1})
$$ 
to obtain an exact sequence
$$
0 \lra M_{\lda} \lra E \lra V \lra 0. 
$$
from \eqref{exact3}, giving rise to 
an element of $\Cal L(V, M_{\lda})$. This establishes a one-to-one 
correspondence between $F_{m,\mbf d,\lda}(p)$ and the 
$\scr P(V; U_1,\ldots,U_l)\times \Aut(M_{\lda})$-orbits on a subset 
$\Cal L'(V,M_{\lda}) \subseteq \Cal L(V,M_{\lda})$. Here, $\Cal L'(V,M_{\lda})$ is the set of those elements of $\Cal L(V,M_{\lda})$ that are given by  
exact sequences
$$
0 \lra M_{\lda} \lra E \stackrel{\pi}{\lra} V \lra 0 
$$
such that $\pi^{-1}(U_{l-1})$ is central in $E$. (This condition corresponds to condition (\ref{centcon}) of the definition of $F_{m,\mbf d,\lda}(p)$.)  
If $B$ is a finite $\mZ_p$-module, let
$$
\Cal F(V,U_{l-1},B)=
\Hom\left(\bigwedge\nolimits^2 \left( \frac{V}{U_{l-1}} \right), B[p] \right) 
\oplus \left( V^* \otimes \frac{B}{pB} \right). 
$$ 
Using the isomorphism \eqref{Lnat2}, we deduce the following result.

\begin{prop}\label{corr2} There is a one-to-one correspondence between 
$F_{m,\mbf d,\lda}(p)$ and the orbits of
$\scr P(V;U_1,\ldots,U_{l-1})\times \Aut(M_{\lda})$ in 
$\Cal F(V,U_{l-1},M_{\lda})$.
\end{prop} 

The following lemma is key to proving Theorem \ref{Lie}.
\begin{lem}\label{key} Let $\lda$ be a partition and $m\in \mN$. Let 
$\mbf d=(d_1,\ldots,d_l)$ be a tuple of nonnegative integers such that 
$m=d_1+\cdots+d_l$ and $d_i>0$ for $i\in [1,l-1]$. 
Then $|F_{m,\mbf d,\lda}(p)|$ is a PORC function of $p$.    
\end{lem} 

\begin{proof} Let $G=\scr P(V;U_1,\ldots,U_{l-1})\times \Aut(M_{\lda})$. 
By Proposition~\ref{corr2}, it is enough to show that 
$
\gamma(G, \Cal F(V,U_{l-1},M_{\lda}))
$
is a PORC function of $p$. The group $\GL(V/U_{l-1})\times \GL(M_{\lda}[p])$ 
acts by linear maps on 
$\Hom \left(\bigwedge^2 (V/U_{l-1}), M_{\lda}[p] \right)$, and 
the group $\GL(V)\times \GL(M_{\lda}/p M_{\lda})$ acts on 
$V^* \otimes (M_{\lda}/p M_{\lda})$. 
Combining these actions, we obtain a homomorphism
$$
\Gamma: \GL(V)\times \GL(V/U_{l-1}) \times \GL(M_{\lda}[p])
\times \GL(M_{\lda}/p M_{\lda}) 
\lra \GL(\Cal F(V,U_{l-1},M_{\lda})).  
$$ 
It is easy to see that $\Gamma$ gives rise to an algebraic family of groups.
Moreover, the action of $G$ on $\Cal F(V,U_{l-1},M_{\lda})$ is given by
$\Gamma \circ \Psi$ where 
$$
\Psi: G \lra \GL(V)\times \GL(V/U_{l-1}) \times \GL(M_{\lda}[p])
\times \GL(M_{\lda}/p M_{\lda})
$$
is the natural homomorphism. Thus, by Theorem \ref{Hig}, 
it suffices to prove that $\im \Psi$ is a uniform family of subgroups, as 
$p$ varies.
$\Psi$ splits into the obvious maps
$$
\alpha: \scr P(V; U_1,\ldots,U_{l-1}) \lra \GL(V)\times \GL(V/U_{l-1}) \quad \text{and } 
$$  
$$
\beta: \Aut(M_{\lda}) \lra \GL(M_{\lda}[p]) \times \GL(M_{\lda}/p M_{\lda}). 
$$
The image of $\beta$ is a uniform family of subgroups by 
Proposition \ref{uniB}. The image of $\alpha$ is a uniform family of subgroups
by Proposition \ref{Hall:gen} and Lemma \ref{subproj}. The image of $\Psi$ 
is the direct product of $\im\alpha$ and $\im\beta$. By Lemma \ref{uniprod},
$\im\Psi$ is a uniform family of subgroups.
\end{proof}

We are now in a position to finish the proof of Theorem \ref{Lie}. 
Let $X_{m,\mbf d,\lda}(p)$ be the
number of elements $[(E,B; W_1,\ldots,W_{l-1})]$ of $F_{m,\mbf d,\lda}(p)$ such that
$W_{l-1}=Z(E)$. Using reverse induction on $l$, we shall show that
$X_{m,\mbf d, \lda}(p)$ is a PORC function of $p$.

The base case is $l=m+1$. Then $\mbf d=(1,1,\ldots,1,0)$, so $W_{l-1}=E=Z(E)$. 
Hence, $X_{m,\mbf d,\lda}(p)=|F_{m,\mbf d,\lda}(p)|$ is PORC by Lemma \ref{key}.  

In the general case, for a fixed $k\in [1,d_l]$, elements $[(E,B; W_1,\ldots,W_l)]$ of 
$F_{m,\mbf d,\lda}$ with $\dim(Z(E)/W_{l-1})=k$ 
are in a one-to-one correspondence with 
the elements 
$$
[(E,B; W_1,\ldots,W_l)]\in  F_{m,(d_1,\ldots,d_{l-1},k,d_l-k),\lda}
$$ 
 satisfying $Z(E)=W_{l}$. (To obtain the correspondence, put 
$W_l=Z(E)$.)
 Hence, 
$$
X_{m,\mbf d,\lda}(p)=|F_{m,\mbf d,\lda}(p)|-
\sum_{k=1}^{d_l} X_{m,(d_1,\ldots,d_{l-1},k,d_l-k),\lda}(p). 
$$
The right-hand side is PORC by Lemma \ref{key} and the inductive hypothesis.
Therefore, the left-hand side is a PORC function of $p$ as well. Putting $l=1$, we obtain the following result.

\begin{thm} For any natural number $m$ and any partition $\lda$, the number of isomorphism classes of Lie algebras E over $\mZ_p$ of nilpotency class $2$ such that $E/Z(E)$ is elementary abelian of dimension $m$ and $Z(E)$ has elementary divisor type $\lda$, 
is a PORC function of $p$.   
\end{thm}

Theorem \ref{Lie} follows by summing over the possible values of $m$ and $\lambda$.    

\bibliographystyle{amsplain}
\bibliography{porc}

\end{document}